\documentclass[12pt,reqno]{amsart}

\usepackage{amsmath,amssymb,amsthm,amsfonts}
\usepackage{bbm}
\usepackage{graphicx}
\allowdisplaybreaks
\usepackage{xcolor}
\usepackage{fullpage}
\usepackage[linktoc=all,colorlinks,linkcolor=teal,citecolor=cyan]{hyperref}
\usepackage{cleveref}
\numberwithin{equation}{section}
\usepackage{comment}
\usepackage{bm}
\usepackage{xcolor}
\newcommand{\mb}[1]{\mathbb{#1}}

\newcommand{\st}[1]{\substack{#1}}

\newcommand{\lr}[1]{\left(#1\right)}

\begin{document}
\numberwithin{equation}{section}
\newtheorem{thm}{Theorem}
\newtheorem{algo}{Algorithm}
\newtheorem{lem}{Lemma}[section]
\newtheorem{de}{Definition} 
\newtheorem{ex}{Example}
\newtheorem{pr}{Proposition} 
\newtheorem{claim}{Claim} 
\newtheorem{re}{Remark}
\newtheorem{co}{Corollary}
\newtheorem{conv}{Convention}
\newcommand{\di}{\hspace{1.5pt} \big|\hspace{1.5pt}}
\newcommand{\idi}{\hspace{.5pt}|\hspace{.5pt}}
\newcommand{\hs}{\hspace{1.3pt}}
\newcommand{\thmf}{Theorem~1.15$'$}
\newcommand{\ndi}{\centernot{\big|}}
\newcommand{\nidi}{\hspace{.5pt}\centernot{|}\hspace{.5pt}}
\newcommand{\lp}{\mbox{$\hspace{0.12em}\shortmid\hspace{-0.62em}\alpha$}}
\newcommand{\PQ}{\bb{P}^1(\bb{Q})}
\newcommand{\pmn}{\cl{P}_{M,N}}
\newcommand{\lcm}{\operatorname{lcm}}
\newcommand{\he}{holomorphic eta quotient\hspace*{2.5pt}}
\newcommand{\hes}{holomorphic eta quotients\hspace*{2.5pt}}
\newcommand{\defG}{Let $G\subset\GG$ be a subgroup that is conjugate to a finite index subgroup of $\G$. } 
\newcommand{\defg}{Let $G\subset\GG$ be a subgroup that is conjugate to a finite index subgroup of $\G$\hs\hs} 
\renewcommand{\phi}{\varphi}
\newcommand{\Z}{\bb{Z}}
\newcommand{\ZD}{\Z^{\D}}
\newcommand{\N}{\bb{N}}
\newcommand{\Q}{\bb{Q}}
\newcommand{\pii}{{{\pi}}}
\newcommand{\R}{\bb{R}}
\newcommand{\C}{\bb{C}}
\newcommand{\I}{\hs\cl{I}_{n,N}}
\newcommand{\SL}{\operatorname{SL}_2(\Z)}
\newcommand{\St}{\operatorname{Stab}}
\newcommand{\D}{\cl{D}_N}
\newcommand{\rh}{{{\boldsymbol\rho}}}
\newcommand{\bh}{{\cl{M}}} 
\newcommand{\lv}{\hyperlink{level}{{\text{level}}}\hspace*{2.5pt}}
\newcommand{\fct}{\hyperlink{factor}{{\text{factor}}}\hspace*{2.5pt}}
\newcommand{\q}{\hyperlink{q}{{\mathbin{q}}}}
\newcommand{\rd}{\hyperlink{redu}{{{\text{reducible}}}}\hspace*{2.5pt}}
\newcommand{\ird}{\hyperlink{irredu}{{{\text{irreducible}}}}\hspace*{2.5pt}}
\newcommand{\str}{\hyperlink{strong}{{{\text{strongly reducible}}}}\hspace*{2.5pt}}
\newcommand{\rdn}{\hyperlink{redon}{{{\text{reducible on}}}}\hspace*{2.5pt}}
\newcommand{\atl}{\hyperlink{atinv}{{\text{Atkin-Lehner involution}}}\hspace*{3.5pt}}
\newcommand{\atls}{\hyperlink{atinv}{{\text{Atkin-Lehner involutions}}}\hspace*{3.5pt}}
\newcommand{\T}{\mathrm{T}}
\renewcommand{\H}{\fr{H}}
\newcommand{\W}{\text{\calligra W}_n}
\newcommand{\GG}{\cl{G}}
\newcommand{\g}{\fr{g}}
\newcommand{\Gm}{\Gamma}
\newcommand{\Gmtl}{\widetilde{\Gamma}_\ell}
\newcommand{\gm}{\gamma}
\newcommand{\go}{\gamma_1}
\newcommand{\gmt}{\widetilde{\gamma}}
\newcommand{\gmdt}{\widetilde{\gamma}'}
\newcommand{\gmot}{\widetilde{\gamma}_1}
\newcommand{\gmdot}{{\widetilde{\gamma}}'_1}
\newcommand{\s}{\Large\text{{\calligra r}}\hspace{1.5pt}}
\newcommand{\ms}{m_{{{S}}}}
\newcommand{\nisim}{\centernot{\sim}}
\newcommand{\level}{\hyperlink{level}{{\text{level}}}}
\newcommand{\Redcon}{the \hyperlink{red}{\text{Reducibility~Conjecture}}}
\newcommand{\ReDcon}{The \hyperlink{red}{\text{Reducibility~Conjecture}}}
\newtheorem*{ThmA}{Theorem A}
\newtheorem{Conj}{Conjecture}
\newtheorem*{ThmB}{Theorem B}
\newtheorem*{ThmC}{Theorem C}
\newtheorem*{lmB}{Lemma B}
\newtheorem*{CorA}{Corollary A}
\newtheorem*{CorB}{Corollary B}
\newtheorem*{CorC}{Corollary C}
\newcommand{\Conred}{Conjecture~$1'$}
\newcommand{\effth}{Theorem~$\ref{27.11.2015}'$}
\newcommand{\Conredd}{Conjecture~$1''$}
\newcommand{\Conreddd}{Conjecture~$1'''$}
\newcommand{\Conired}{Conjecture~$2'$}
\newtheorem*{pro}{\textnormal{\textit{Proof of Lemma~\ref{27.11.2015.1}}}}
\newtheorem*{cau}{Caution}
\newtheorem{thrmm}{Theorem}[section]
\newtheorem*{thmA}{Theorem~A}
\newtheorem*{corA}{Corollary~A}
\newtheorem*{corB}{Corollary~B}
\newtheorem*{corC}{Corollary~C}
\newtheorem{no}{Notation}
\renewcommand{\thefootnote}{\fnsymbol{footnote}}
\newtheorem{oq}{Open Question}
\newtheorem{hy}{Hypothesis} 
\newtheorem{expl}{Example}
\newcommand\ileg[2]{\bigl(\frac{#1}{#2}\bigr)}
\newcommand\leg[2]{\Bigl(\frac{#1}{#2}\Bigr)}
\newcommand{\e}{\eta}
\newcommand{\sgn}{\operatorname{sgn}}
\newcommand{\bb}{\mathbb}
\newcommand{\fr}{\mathfrak}
\newcommand{\cl}{\mathcal}
\newcommand{\rad}{\mathrm{rad}}
\newcommand{\ord}{\operatorname{ord}}

\title[Diophantine approximation with sums of two squares]{Diophantine approximation with sums of two squares II}

\author{Stephan Baier}
\address{Stephan Baier,
Ramakrishna Mission Vivekananda Educational and Research Institute, Department of Mathematics, G. T. Road, PO Belur Math, Howrah, West Bengal 711202, India}
\email{stephanbaier2017@gmail.com}
\author{Habibur Rahaman}
\address{Habibur Rahaman, Indian Institute of Science Education \& Research Kolkata,
Department of Mathematics and Statistics, Mohanpur, West Bengal 741246, India}
\email{hr21rs044@iiserkol.ac.in}


\maketitle
\begin{abstract}
In \cite{B-R}, we showed that for every irrational number $\alpha$, there exist infinitely many positive integers $n$ represented by any given positive definite binary quadratic form $Q$, satisfying $||\alpha n||<n^{-(1/2-\varepsilon)}$ for any fixed $\varepsilon>0$. We also provided a quantitative version with a lower bound when the exponent $1/2-\varepsilon$ is replaced by a smaller exponent $\gamma<3/7-\varepsilon$. In this article, we establish a  quantitative version for the exponent $1/2-\varepsilon$, where we confine ourselves to the particular case of sums of two squares. 
\end{abstract}
\section{Introduction and main results}
Let $\|x\|$ be the distance of $x\in\mb{R}$ to the nearest integer and $\alpha\in \mathbb{R}$ be an arbitrary but fixed irrational number. 
The classical Dirichlet approximation theorem implies that there are infinitely many positive integers $n$ such that $$\|\alpha n\|<n^{-1}.$$ A natural question to ask is what happens if we restrict $n$ to a sparse subset $\mathcal{A}$ of the set of positive integers $\mb{N}$ with interesting arithmetic properties. There are many results regarding this question in the literature. For example, the best known result regarding the case when $\mathcal{A}$ equals the set of primes is that $\|\alpha p\|<p^{-(1/3-\varepsilon)}$ holds for infinitely many primes $p$ and any fixed $\varepsilon>0$. This is due to Matom\"aki \cite{Matomaki}. For the set of squarefree numbers, we have a much better result, namely that $\|\alpha n\|<n^{-(2/3-\varepsilon)}$ holds for infinitely many squarefree $n$, due to Heath-Brown\cite{H-B}. There are also some recent results on $n$ varying over smooth numbers (see \cite{Yau}, \cite{Baker}).

Another interesting case to consider is that of the set of integers represented by a sum of two squares, or, more generally, the set of integers represented by a fixed positive definite binary quadratic form (PBQF). In \cite{Balog}, Balog and Perelli proved $\|\alpha n\|<n^{-(1/2-\varepsilon)}$ (improved by Heath-Brown, as mentioned above) for infinitely many squarefree numbers $n$ and stated without proof that the same exponent $1/2-\varepsilon$ works for $n$ varying over integers that are sums of two squares, using the same method as for squarefree numbers.  In \cite{B-R}, the authors of the current paper studied this problem and pointed out that the method of Balog and Perelli actually gives no more than the exponent $1/3-\varepsilon$ for sums of two squares. We then recovered the exponent $1/2-\varepsilon$ from an old result of Cook \cite{Cook}. In fact, we proved the following more general result for integers represented by binary quadratic forms. 
\begin{thm}\cite[Theorem~1]{B-R}\label{thm1}
Let $Q(x,y)=a_1x^2+b_1xy+c_1y^2$ be a positive definite integral binary quadratic form. Let 
\begin{align*}
    \mathcal{A}_Q:=\{n\in\mathbb{N}: Q(x,y)=n \text{ for some } x,y\in\mathbb{Z}\}.
\end{align*}
 Let $\alpha$ be any fixed irrational number. Then, there are infinitely many integers $n\in\mathcal{A}_Q$ such that 
     \begin{align} \label{Cookin}
         \|\alpha n\|<n^{-1/2+\varepsilon},
     \end{align}
     for any fixed $\varepsilon>0$.
\end{thm}

The method in \cite{Cook} does not yield a lower bound for the numbers of $n \in \mathcal{A}_Q$ satisfying $n\le X$ and \eqref{Cookin}. Hence, the above Theorem \ref{thm1}, proved in \cite{B-R}, is a pure existence result. However, in the same paper \cite{B-R}, we also derived a quantitative result on this problem, yielding a lower bound in certain ranges. This came at the cost of a weaker exponent $3/7-\varepsilon$, though. To be precise, we proved the following.

\begin{thm}\cite[Corollary~1]{B-R}\label{thm2}
Let $Q(x,y)=a_1x^2+b_1xy+c_1y^2$ be a positive definite integral binary quadratic form and $\Delta$ be the absolute value of the discriminant of $Q$. Suppose that $q\in \mathbb{N}$ and $a\in \mathbb{Z}$ satisfy $(2\Delta a,q)=1$ and the inequality  
$$
\left|\alpha-\frac{a}{q}\right|< \frac{24\Delta^2}{q^2}.
$$
Assume that $2/5< \beta<1$. Set $X:=q^{1+\beta}$ and $\gamma:=
(1-\beta)/(1+\beta)$. Then
$$
\sharp \{n\in \mathcal{A}_Q: n\le 2X, \ \|n\alpha\|<C_1n^{-\gamma}\} \gg X^{1-\gamma-C_2/(\log\log X)}   
$$
for suitable constants $C_1,C_2>0$. Moreover, there are infinitely many rational numbers $a/q$ satisfying the conditions above. 
\end{thm}

\begin{re} We note that $(1-2/5)/(1+2/5)=3/7$, and thus Theorem \ref{thm2} recovers Theorem~\ref{thm1} with a weaker exponent of $3/7-\varepsilon$ in place of $1/2-\varepsilon$, but coming with a lower bound for the number of $n$'s satisfying $\|n\alpha\|<n^{-3/7+\varepsilon}$. 
\end{re}
In the present paper, we derive a lower bound for the particular case $Q(x,y)=x^2+y^2$ when the exponent equals $1/2-\varepsilon$. We shall first establish the following theorem which states an asymptotic formula for a smoothed and weighted sum.

 \begin{thm}\label{main_thm} Let $q\in \mathbb{N}$ and $a\in \mathbb{Z}$ such that $(2a,q)=1$. Let $L=q^\beta, X=Lq=q^{1+\beta}$ with $1/3< \beta<1$. Let $w: \mathbb{R}\rightarrow \mathbb{R}_{\ge 0}$ and $\Phi:\mathbb{R}\rightarrow \mathbb{R}_{\ge 0}$ be test functions compactly supported in $[1,2] $ and $[-1,1]$, respectively. 
    Define \begin{equation} \label{Sdef}
    S:=\sum\limits_{\substack{b\in\mb{Z}\\(b,q)=1}}\Phi\lr{\frac{b}{L}} \sum\limits_{\substack{n\\ na\equiv b\bmod{q}}} r_2(n)w\lr{\frac{n}{X}}, 
\end{equation}
where $r_2(n)$ is the number of integral representations of $n$ by sums of two squares. Then, we have 
\begin{align*}
    S=  \pi\hat{\Phi}(0)\hat{w}(0)L^2\cdot\frac{\varphi(q)}{q}\cdot \prod_{p|q}\lr{1-\frac{\chi_4(p)}{p}}\cdot \left(1+o(1)\right)
\end{align*}
as $q\rightarrow \infty$, where $\chi_4$ is the non-trivial Dirichlet character modulo $4$.
Moreover, $S$ satisfies the lower bound 
 \begin{align*}
         S\gg \frac{L^2}{(\log\log q)^2}.
     \end{align*}

\end{thm}

Our proof makes use of the well-known formula $r_2(n)=4\sum_{d|n} \chi_4(d)$, Poisson summation at several places and reciprocity for Kloosterman fractions. 
We will use Theorem \ref{main_thm} above to deduce the following lower bound for Diophantine approximation with integers which are sums of two squares, which is close to the expected bound. 

\begin{co}\label{co_main_thm} Suppose that $q\in \mathbb{N}$ and $a\in \mathbb{Z}$ satisfy $(2a,q)=1$ and the inequality  
$$
\left|\alpha-\frac{a}{q}\right|< \frac{24}{q^2}.
$$
Fix $\varepsilon>0$ and $X=q^{2/(1+\gamma)}$. Then, for any fixed $0<\gamma<1/2$ and $q\rightarrow \infty$, we have  
\begin{equation}\label{thelower}
\sharp \{X\le n\le 2X: n=x^2+y^2, x,y\in\mb{Z}, \ \|n\alpha\|<C_1n^{-\gamma}\} \gg X^{1-\gamma-C_2/(\log\log X)}   
\end{equation}
for suitable absolute constants $C_1,C_2>0$. Moreover, there are infinitely many rational numbers $a/q$ satisfying the conditions above. 
\end{co}

\begin{re} Since the number $N(X)$ of integers $n\le X$ which are sums of two squares satisfies $N(X)\sim KX/\sqrt{\log X}$ with $K$ the Landau-Ramanujan constant, we expect that the lower bound of correct order of magnitude in \eqref{thelower} should be $\gg X^{1-\gamma}/\sqrt{\log X}$. 
\end{re}

\begin{re}
Our improvement of the exponent from $3/7$ to $1/2$ in the case of the particular quadratic form $Q(x,y)=x^2+y^2$ depends on the above-mentioned convolution formula for $r_2(n)$, which allows us to avoid the use of the Voronoi summation formula. A convolution formula of this kind holds, more generally, for positive definite binary quadratic forms $Q$ with class number 1. Therefore, with some extra efforts, our results should be extendable to such forms. 
\end{re}

\begin{re}
We point out that Corollary~\ref{co_main_thm} is for particular $X$ depending on $q$ and hence depending on $\alpha$. However, this dependence can be removed for a certain class of irrationals $\alpha$. Suppose that $\alpha$ is an irrational number with continued fraction expansion $\alpha=[a_0,a_1,\ldots]$, 
where the partial quotients are bounded, that is, $a_n\le M$ for all $n\in\mathbb{N}\cup \{0\}$. Let $p_n/q_n$ denote the convergents of $\alpha$. Then $q_0=1,\ q_1=a_1$ and $q_n=a_nq_{n-1}+q_{n-2}$ for $ n\ge 2$. Consequently, $q_{n-1}<q_n\le (M+1)q_{n-1}$, for all $ n\ge 1$. Therefore, for every sufficiently large $X>0$, we can always find $q_n$ such that $q_n\asymp X^{(1+\gamma)/2}$. Taking $q=q_n$, Corollary~\ref{co_main_thm} then follows for all sufficiently large $X>0$. There are infinitely many irrational numbers with bounded partial quotients, in particular, all quadratic irrationals have this property.

\end{re}

\subsection{Notations} In this paper, we will use the following standard notations.
\begin{itemize}
    \item By $\varepsilon$, we denote an arbitrarily small positive real number. 
    \item Expressions of the form ${f}(x)=O({g}(x))$, ${f}(x) \ll {g}(x)$, and ${g}(x) \gg {f}(x)$ signify that $|{f}(x)| \leq C|{g}(x)|$ for all sufficiently large $x$, where $C>0$ is an absolute constant. A subscript of the form $\ll_A$ or $O_{A}$ means that the implied constant may depend on the parameter $A$. Throughout, we allow all implied constants to depend on $\varepsilon$.  
\item Expressions of the form $f(x)=o(g(x))$ signify that $\lim_{x\rightarrow\infty} g(x)/f(x)=0$. 
\item $\|x\|$ denotes the distance of the real number $x$ to the nearest integer.
    \item We denote the divisor function by $\tau(n)$, that is, for any $n\in\mb{N}$, $$\tau(n)=\sum_{d|n}1.$$ At several places, we will use the well-known bound $\tau(n)\ll_{\varepsilon} n^{\varepsilon}$ for any $\varepsilon>0$.
    \item The Euler totient function is denoted by $\varphi(n)$. We will use the well-known bound 
$$
\frac{n}{\varphi(n)}\gg \log \log (10n)
$$
at several places. 
\item We denote the number of integral representations of $n\in \mathbb{N}$ as a sum of two squares by $r_2(n)$.
 \item For any set $\mathcal{E}$, we write
    \begin{align*}
        \mathbbm{1}_{\mathcal{E}}(n)=
        \begin{cases}
            1, & n\in \mathcal{E},\\
            0, & n\notin \mathcal{E}.
        \end{cases}
    \end{align*}
 \item For any real number $x$, we write $e(x):=e^{2\pi ix}$.
    \item Given any Schwartz class function $f\colon \mathbb{R}\to \mathbb{C}$, we define its Fourier transform $\hat{f}\colon \mathbb{R}\to \mathbb{C}$ by
    \begin{align*}
        \hat{f}(x)=\int_{\mathbb{R}}f(t)e(-xt)\: dt.
    \end{align*}
(For details on the Schwartz class, see \cite{StSh}. Test functions are Schwartz class functions with compact support taking non-negative real values.) 
\end{itemize}

\noindent{\bf Acknowledgments.}
The first-named author would like to thank the Ramakrishna Mission Vivekananda Educational and Research Institute for an excellent work environment. The research of the second-named author was supported by a Prime Minister Research Fellowship (PMRF ID- 0501972), funded by the Ministry of Education, Govt. of India.

 \section{Preliminaries}

The following lemma generalizes the Dirichlet approximation theorem and implies the infinitude of rational numbers $a/q$ in Corollary \ref{co_main_thm}. 
\begin{lem}\label{lem_Dirichlet_approx}
    Let $d\in \mathbb{N}$ and $\alpha\in \mathbb{R}\setminus\mathbb{Q}$. Then there exist infinitely many pairs $(b,r)\in \mathbb{Z}\times \mathbb{N}$ such that
$$
(r,bd)=1\quad \mbox{and} \quad \left|\alpha-\frac{b}{r}\right|\le \frac{6d^2}{r^2}. 
$$
\end{lem}
\begin{proof}
    See the proof of \cite[Lemma~5]{B-R}.
\end{proof}
We also note down the reciprocity law for Kloosterman fractions.
\begin{lem}\label{lem_rec}
    Let $a,b,c\in\mathbb{Z}$ with $(a,b)=1$. Then 
    \begin{align*}
        \frac{\bar{a}c}{b}\equiv -\frac{\bar{b}c}{a}+\frac{c}{ab} \quad (\bmod{1}),
    \end{align*}
    where $a\bar{a}\equiv 1(\bmod{b})$ and $b\bar{b}\equiv 1 (\bmod{a})$.
\end{lem}
\begin{proof}
     See \cite[Equation~(2.15)]{DFI}.
\end{proof}

We will use the following generalized version of the Poisson summation formula at several places.

\begin{lem}[Poisson summation] \label{Poisson} Let $\Phi:\mathbb{R}\rightarrow \mathbb{C}$ be a Schwartz class function, $L>0$, $q\in \mathbb{N}$, $a\in \mathbb{Z}$ and $\omega\in \mathbb{R}$. Then
\begin{align*}
\sum\limits_{\substack{l\in \mathbb{Z}\\ l\equiv a\bmod{q}}} \Phi\left(\frac{l}{L}\right) e\left(l\omega\right)=\frac{L}{q}\sum\limits_{n\in \mathbb{Z}} \hat{\Phi}\left(L\left(\omega+\frac{n}{q}\right)\right)e\left(\frac{an}{q}\right).
\end{align*}
\end{lem}

\begin{proof} This arises by linear changes of variables from the well-known basic version of the Poisson summation formula which asserts that
$$
\sum\limits_{n\in \mathbb{Z}} F(n)=\sum\limits_{n\in \mathbb{Z}} \hat{F}(n)
$$
for any Schwartz class function $F:\mathbb{R}\rightarrow \mathbb{C}$ (see \cite{StSh}).
\end{proof} 

\section{Proof of Theorem~\ref{main_thm}}
\subsection{Application of a convolution formula}
It is well-known that $r_2(n)$ satisfies the convolution formula
\begin{equation} \label{convo}
r_2(n)=4\sum_{d|n}\chi_4(d),
\end{equation} 
where $\chi_4$ is the quadratic Dirichlet character modulo $4$. Substituting this into \eqref{Sdef} and re-arranging summations, we have 
\begin{align*}
    S=4\sum_{\st{d\in \mathbb{N}\\(d,q)=1}}\chi_4(d)\sum_{\substack{b\in\mb{Z}\\(b,q)=1}}\Phi\lr{\frac{b}{L}}\sum_{\substack{m\in\mb{Z}\\amd \equiv b \bmod{q}}}w\lr{\frac{md}{X}},
\end{align*}
where we note that the coprimality conditions $(a,q)=1$ and $(b,q)=1$ force $md$ to be coprime to $q$ too. For a suitable test function $\Omega$ with compact support in $[1/2,2]$, we have 
$$
\sum\limits_{i=0}^{\infty} \Omega\left(\frac{x}{2^i}\right)= 1 \quad \mbox{if } x\ge 1. 
$$
Hence, using this smooth partition of unity in the $d$-sum, we can write 
\begin{align*}
    S=\sum_{i=0}^{\infty} S(2^i),
\end{align*}
where for $D=2^i$, 
\begin{align}\label{def-S(D)}
    S(D):=4\sum_{\st{d\in\mb{N}\\(d,q)=1}}\Omega\lr{\frac{d}{D}}\chi_4(d)\sum_{\substack{b\in\mb{Z}\\(b,q)=1}}\Phi\lr{\frac{b}{L}}\sum_{\substack{m\in\mb{Z}\\amd \equiv b \bmod{q}}}w\lr{\frac{md}{X}}.
\end{align}
We divide $S$ into two subsums accounting for the cases $D\le \sqrt{X}$ and $D>\sqrt{X}$, that is,
\begin{align}\label{S=S1+S2}
    S=S_1+S_2,
\end{align}
where 
\begin{align}\label{def-S1}
     S_1=\sum_{D\leq \sqrt{X}}S(D)
\end{align}
and 
\begin{align}\label{def-S2}
     S_2=\sum_{\sqrt{X}<D\le 2X}S(D),
\end{align}
with $D$ running over the powers $2^i$ ($i=0,1,2,...$), respectively. 
\subsection{Division of $S_1$ into main and error terms} Since $w$ has compact support in $[1,2]$, only the terms with $X\le md\le 2X$ contribute. Hence, if $d$ is small then the relevant $m$-range $X/d\le m\le 2X/d$ is long, and it is therefore beneficial to use the Poisson summation formula, Lemma~\ref{Poisson}, for the $m$-sum, giving
\begin{equation*}
\begin{split}
\sum_{\substack{m\in\mb{Z}\\amd \equiv b \bmod{q}}}w\lr{\frac{md}{X}}=\sum_{\substack{m\in\mb{Z}\\m \equiv b\overline{a}\overline{d} \bmod{q}}}w\lr{\frac{md}{X}} = &  \frac{X}{qd}\sum\limits_{h\in \mathbb{Z}} \hat{w}\left(h\cdot \frac{X}{qd}\right)e\left(h\cdot \frac{b \bar{a}\bar{d}}{q}\right),
\end{split}
\end{equation*}
where $\bar{a}$ and $\bar{d}$ are suitable multiplicative inverses of $a$ and $d$ modulo $q$, respectively.
Plugging this into \eqref{def-S(D)} and interchanging the summations, we obtain
\begin{equation*}
S(D)=\frac{4X}{q}\sum\limits_{\substack{d\in\mb{N}\\ (d,q)=1}}\Omega\lr{\frac{d}{D}} \frac{\chi_4(d)}{d} \sum\limits_{h\in\mathbb{Z}} \hat{w}\left(h\cdot \frac{X}{qd}\right) \sum\limits_{\substack{b\in\mb{Z}\\(b,q)=1}}\Phi\lr{\frac{b}{L}} e\left(b\cdot \frac{hc\bar{d}}{q}\right),
\end{equation*}
where $\bar{a}\equiv c \bmod{q}$ with $1\le c\leq q$.
Recall that $X=qL$. We substitute this into \eqref{def-S1} and isolate a main term accounting for the contribution of $h=0$, which is 
\begin{equation} \label{S3eq}
M:=4L \hat{w}(0)\cdot  \bigg(\sum\limits_{\substack{b\in\mb{Z}\\(b,q)=1}}\Phi\lr{\frac{b}{L}}\bigg)\cdot \bigg( \sum\limits_{\substack{d\in\mb{N}\\ (d,q)=1}}\frac{\chi_4(d)}{d} \sum_{D\le \sqrt{X}}\Omega\lr{\frac{d}{D}}\bigg).
\end{equation}
Hence, 
\begin{equation} \label{splitting}
S_1=M+E,
\end{equation}
where 
\begin{align}\label{S4eq}
    E:=\sum_{D\le \sqrt{X}}E(D)
\end{align}
with 
\begin{equation*}
E(D)=\frac{4X}{q} \sum\limits_{\substack{d\in\mb{N}\\ (d,q)=1}}\Omega\lr{\frac{d}{D}} \frac{\chi_4(d)}{d} \sum\limits_{\st{h\in\mathbb{Z}\\h\neq 0}} \hat{w}\left(h\cdot \frac{X}{qd}\right) \sum\limits_{\substack{b\in\mb{Z}\\(b,q)=1}}\Phi\lr{\frac{b}{L}} e\left(b\cdot \frac{hc\bar{d}}{q}\right). 
\end{equation*}
Using the triangle inequality and the rapid decay of $\hat{w}$ to truncate the $h$-sum at $|h|\le X^{\varepsilon}qD/X$ at the cost of a negligible error, it follows that 
\begin{equation} \label{S1D}
E(D)\ll \frac{X}{qD} \sum\limits_{\substack{D/2\le d\le 2D\\ (d,q)=1}} \ \sum\limits_{0<|h|\le X^{\varepsilon}qD/X} \Big| \sum\limits_{\substack{b\in\mb{Z}\\(b,q)=1}}\Phi\lr{\frac{b}{L}} e\left(b\cdot \frac{hc\bar{d}}{q}\right) \Big|+O\left(X^{-2026}\right).
\end{equation}

\subsection{Approximation of the main term}
Next, we derive an asymptotic formula for the main term $M$.
First, we remove the coprimality condition in the $b$-sum on the right-hand side of \eqref{S3eq} using M\"obius inversion, writing
$$
\sum\limits_{\substack{b\in\mb{Z}\\(b,q)=1}}\Phi\lr{\frac{b}{L}}=\sum\limits_{t|q} \mu(t) \sum\limits_{\substack{b\in\mb{Z}}}\Phi\lr{\frac{bt}{L}}.
$$
We truncate the $t$-sum on the right-hand side at $t\le LX^{-\varepsilon}$: Using the rapid decay of $\Phi$ and the divisor bound $\tau(q)\ll q^{\varepsilon}$, we have
$$
\sum\limits_{\substack{b\in\mb{Z}\\(b,q)=1}}\Phi\lr{\frac{b}{L}}=
\sum\limits_{\substack{t|q\\ t\le LX^{-\varepsilon}}} \mu(t) \sum\limits_{\substack{b\in\mb{Z}}} \Phi\lr{\frac{bt}{L}}+
O\left(X^{2\varepsilon}\right). 
$$
Applying the Poisson summation formula, Lemma \ref{Poisson}, to the $b$-sum on the right-hand side, and using the rapid decay of $\hat{\Phi}$, we see that
$$
\sum\limits_{\substack{b\in\mb{Z}}}\Phi\lr{\frac{bt}{L}}=\frac{L}{t} \sum\limits_{n\in\mb{Z}}\hat\Phi\lr{\frac{nL}{t}} = \frac{L}{t} \cdot\hat\Phi(0)+O\left(X^{-2026}\right) \quad \mbox{if } t\le LX^{-\varepsilon}.
$$
Combining the above equations and using the divisor bound again, we obtain
\begin{equation*}
\begin{split}
\sum\limits_{\substack{b\in\mb{Z}\\(b,q)=1}}\Phi\lr{\frac{b}{L}} = & \hat{\Phi}(0)L \sum\limits_{\substack{t|q\\ t\le LX^{-\varepsilon}}} \frac{\mu(t)}{t} +O\left(X^{2\varepsilon}\right)\\
= & \hat{\Phi}(0)L\sum\limits_{t|q} \frac{\mu(t)}{t} +O\left(X^{2\varepsilon}\right)\\
= & \hat{\Phi}(0)L\cdot \frac{\varphi(q)}{q}+O\left(X^{2\varepsilon}\right).
\end{split}
\end{equation*}
Plugging this into \eqref{S3eq}, we deduce that 
\begin{equation} \label{Mdeduce}
M=4 \hat\Phi(0)\hat{w}(0) L\left(\frac{\varphi(q)}{q}\cdot L +O\left(X^{2\varepsilon}\right)\right)   \sum\limits_{\substack{d\in\mb{N}\\ (d,q)=1}}\frac{\chi_4(d)}{d} \sum_{D\le \sqrt{X}}\Omega\lr{\frac{d}{D}} +O\left(X^{-2026}\right).
\end{equation}

Let $D=2^j$ be the largest power of $2$ such that $2^j\le \sqrt{X}$. Then it follows that 
$$
\Psi(x):=\sum_{D\le \sqrt{X}}\Omega\lr{\frac{x}{D}}=\sum\limits_{i=0}^j\Omega\lr{\frac{x}{2^i}}
$$
is a smooth function such that 
\begin{equation*} 
\Psi(x)=\begin{cases} 1 & \mbox{ if }1\le  x\le 2^j, \\ 0 & \mbox{ if } x\ge 2^{j+1} \end{cases} \quad \mbox{and} \quad \Psi'(x)\ll x^{-1} \mbox{ if } x\ge 1.
\end{equation*}
Hence, we may write
\begin{equation}\label{moreover}
\sum\limits_{\substack{d\in\mb{N}\\ (d,q)=1}}\frac{\chi_4(d)}{d} \sum_{D\le \sqrt{X}}\Omega\lr{\frac{d}{D}}=  \sum\limits_{\substack{d\le 2^j\\ (d,q)=1}} \frac{\chi_4(d)}{d} +\sum\limits_{\substack{2^j<d\le 2^{j+1}\\ (d,q)=1}} \frac{\chi_4(d)}{d} \cdot \Psi(d),
\end{equation}
and partial summation gives
\begin{equation} \label{partial}
\sum\limits_{\substack{2^j<d\le 2^{j+1}\\ (d,q)=1}} \frac{\chi_4(d)}{d} \cdot \Psi(d) =-\int\limits_{2^j}^{2^{j+1}} \Psi'(x) \bigg(\sum\limits_{\substack{2^j<d\le x\\ (d,q)=1}} \frac{\chi_4(d)}{d}\bigg)dx \ll 
\frac{1}{2^j} \int\limits_{2^j}^{2^{j+1}} \bigg| \sum\limits_{\substack{2^j<d\le x\\ (d,q)=1}} \frac{\chi_4(d)}{d} \bigg| dx.
\end{equation}

Now let
    \begin{align*}
        L(1,\chi_4):=\sum_{d\geq 1}\frac{\chi_4(d)}{d}.
    \end{align*}
    It is well-known that $L(1,\chi_4)=\pi/4$, which implies that
    \begin{align}\label{partialsum}
        \sum_{d\le s}\frac{\chi_4(d)}{d}=\frac{\pi}{4}+O\lr{\frac{1}{s}}\quad \mbox{for any } s\ge 1
    \end{align}
since these sums are alternating. 
Using M\"obius inversion again, for any $t\ge 1$, we write 
\begin{align*}
    \sum\limits_{\substack{1\le d\le t\\ (d,q)=1}} \frac{\chi_4(d)}{d}=&\sum\limits_{\substack{1\le d\le t}} \lr{\sum_{k|(d,q)}\mu(k)}\frac{\chi_4(d)}{d}
    =\sum_{k|q}\frac{\mu(k)\chi_4(k)}{k}\sum_{1\le d\le t/k}\frac{\chi_4(d)}{d}.
\end{align*}
Approximating the innermost sum above using \eqref{partialsum}, it follows that 
\begin{align} \label{innerm}
 \sum\limits_{\substack{1\le d\le t\\ (d,q)=1}} \frac{\chi_4(d)}{d}=\frac{\pi}{4}\prod_{\st{p|q}}\lr{1-\frac{\chi_4(p)}{p}}+O\lr{\frac{\tau(q)}{t}} \quad \mbox{for any } t\ge 1.
\end{align}
Consequently,
\begin{equation} \label{conseq}
\sum\limits_{\substack{t< d\le T\\ (d,q)=1}} \frac{\chi_4(d)}{d}\ll \frac{\tau(q)}{t} \quad \mbox{for any } T>t\ge 1.
\end{equation}
Combining \eqref{moreover}, \eqref{partial}, \eqref{innerm} and \eqref{conseq}, we deduce that
\begin{equation*}
\sum\limits_{\substack{d\in\mb{N}\\ (d,q)=1}}\frac{\chi_4(d)}{d}\sum_{D\le \sqrt{X}}\Omega\lr{\frac{d}{D}}=\frac{\pi}{4}
\prod_{p|q}\lr{1-\frac{\chi_4(p)}{p}}+O\left(\frac{\tau(q)}{\sqrt{X}}\right).
\end{equation*}
Plugging this into \eqref{Mdeduce} and using the facts that 
\begin{equation*} 
\begin{split}
 \frac{1}{\log\log(10q)}\ll & \frac{\varphi(q)}{q}=\prod_{p|q}\lr{1-\frac{1}{p}}\le \prod_{p|q}\lr{1-\frac{\chi_4(p)}{p}}\\
\leq & \prod_{p|q}\lr{1+\frac{1}{p}}\le \frac{q}{\varphi(q)}\ll \log\log(10q)
\end{split}
\end{equation*}
and $\tau(q)\ll X^{\varepsilon}$, we arrive at the following lemma. 

\begin{lem}\label{lem_S3}
Let $M$ be given as in \eqref{S3eq}. Then
     $$M=\pi\hat{\Phi}(0)\hat{w}(0)L^2\cdot\frac{\varphi(q)}{q}\cdot \prod_{p|q}\lr{1-\frac{\chi_4(p)}{p}}+O\lr{X^{2\varepsilon}\left(L+\frac{L^2}{\sqrt{X}}\right)}.
$$
     In particular, we have 
     \begin{align*}
         M\gg \frac{L^2}{(\log\log q)^2}.
     \end{align*}
\end{lem}

\subsection{Evaluation of the error term}\label{sec-error}
Next, we show that $E$, defined in \eqref{S4eq}, gives an error contribution.
Using M\"obius inversion, we write the innermost sum on the right-hand side of \eqref{S1D} as
$$
\sum\limits_{\substack{b\in\mb{Z}\\(b,q)=1}}\Phi\lr{\frac{b}{L}} e\left(h\cdot \frac{bc\bar{d}}{q}\right)= 
\sum\limits_{g|q} \mu(g)\sum\limits_{b\in\mb{Z}}\Phi\lr{\frac{bg}{L}} e\left(b\cdot \frac{ghc\bar{d}}{q}\right)
$$
and note that 
$$
\sum\limits_{b\in\mb{Z}}\Phi\lr{\frac{bg}{L}} e\left(b\cdot \frac{ghc\bar{d}}{q}\right) \ll 1+\frac{L}{g}
$$
by a trivial estimate. 
Thus, truncating $E(D)$ at $g\le G$, where $G$ is a free parameter in the range $1\le G<LX^{-\varepsilon}$,  and bounding the remaining sum trivially, it follows from \eqref{S1D} that 
\begin{align}\label{E(D)-bound}
E(D)\ll \frac{X}{qD} \sum\limits_{\substack{D/2\le d\le 2D\\ (d,q)=1}} \ \sum\limits_{0<|h|\le X^{\varepsilon}qD/X} \Big| \sum_{\st{g|q\\g\le G}}\sum\limits_{\substack{b\in\mb{Z}}}\Phi\lr{\frac{bg}{L}} e\left(b\cdot \frac{ghc\bar{d}}{q}\right) \Big|+O\lr{\frac{X^{2\varepsilon}DL}{G}}.
\end{align}
We apply Poisson summation, Lemma \ref{Poisson}, to deduce that
$$
\sum\limits_{b\in\mb{Z}}\Phi\lr{\frac{bg}{L}} e\left(b\cdot \frac{ghc\bar{d}}{q}\right)= \frac{L}{g} \sum\limits_{v\in \mathbb{Z}} \hat\Phi\left(\frac{L}{g}\cdot \left(v+\frac{ghc\bar{d}}{q}\right)\right).
$$ 
Since $(d,q)=1$, by Lemma~\ref{lem_rec} we can write  
$$
\frac{ghc\bar{d}}{q}\equiv -\frac{ghc\bar{q}}{d}+\frac{ghc}{qd}\bmod{1},
$$
where $\bar{q}$ is a multiplicative inverse of $q$ modulo $d$. It follows that
$$
\sum\limits_{b\in\mb{Z}}\Phi\left(\frac{bg}{L}\right) e\left(b\cdot \frac{ghc\bar{d}}{q}\right)= \frac{L}{g}\sum\limits_{v\in \mathbb{Z}} \hat\Phi\left(\frac{L}{g}\cdot \left(v-\frac{ghc\bar{q}}{d}\right)+\frac{Lhc}{qd}\right).
$$
Substituting this into \eqref{E(D)-bound}, we have 
\begin{align}\label{S1_cut2}
E(D)&\ll \frac{LX}{qD}\sum\limits_{\substack{D/2\le d\le 2D\\ (d,q)=1}} \sum\limits_{0<|h|\le X^{\varepsilon}qD/X} \Big| \sum\limits_{\substack{g|q\\g\le G}}\frac{ \mu(g)}{g} \sum\limits_{v\in \mathbb{Z}} \hat\Phi\left(\frac{L}{g}\cdot \left(v-\frac{ghc\bar{q}}{d}\right)+\frac{Lhc}{qd}\right)\Big|\nonumber\\&+O\lr{\frac{X^{2\varepsilon}DL}{G}}.
\end{align}
For $D/2\le d\le 2D$ and $|h|\le X^{\varepsilon}qD/X$, we note that 
\begin{equation} \label{holds}
\left|\frac{Lhc}{qd}\right|\le \frac{Lh}{d}\le \frac{2X^\varepsilon Lq}{X}=2X^{\varepsilon}. 
\end{equation}
Now using the rapid decay of $\hat\Phi$ and $1\le g\le G\le LX^{-\varepsilon}$, we see that
$$
\sum\limits_{v\in \mathbb{Z}} \hat\Phi\left(\frac{L}{g}\cdot \left(v-\frac{ghc\bar{q}}{d}\right)+\frac{Lhc}{qd}\right)\ll X^{-2026}
$$
unless 
\begin{equation} \label{Ydef}
\left\|\frac{ghc\bar{q}}{d}\right\|\le \frac{X^{2\varepsilon}g}{L}=:Y(g).
\end{equation}
Here we note that if $g\le LX^{-\varepsilon}$, then the points 
$$
\frac{L}{g}\cdot \left(v-\frac{ghc\bar{q}}{d}\right)+\frac{Lhc}{qd}, \quad v\in \mathbb{Z}
$$
are $X^{\varepsilon}$-spaced, and under the conditions \eqref{holds} and
$$
\left\|\frac{ghc\bar{q}}{d}\right\|> Y(g),
$$
we have
$$
\min\limits_{v\in \mathbb{Z}} \frac{L}{g}\cdot \left(v-\frac{ghc\bar{q}}{d}\right)+\frac{Lhc}{qd}\ge \frac{L}{g}\cdot Y(g)-2X^{\varepsilon}\gg X^{\varepsilon}.
$$
If \eqref{Ydef} is satisfied, then a trivial estimate gives
$$
\sum\limits_{v\in \mathbb{Z}} \hat\Phi\left(\frac{L}{g}\cdot \left(v-\frac{ghc\bar{q}}{d}\right)+\frac{Lhc}{qd}\right)\ll 1,
$$
again using $g\le LX^{-\varepsilon}$.
By the above considerations, it follows from \eqref{S1_cut2} that
\begin{equation*}
E(D)\ll \frac{LX}{qD}\sum\limits_{D/2\le d\le 2D}\ \sum\limits_{0<|h|\le X^{\varepsilon}qD/X} \sum\limits_{\substack{g|q\\g\le G}} \frac{1}{g} \cdot \mathbbm{1}_{[0,Y(g)]}\left(\left\|\frac{ghc\bar{q}}{d}\right\|\right)+O\lr{\frac{X^{2\varepsilon}DL}{G}},
\end{equation*}
where $\mathbbm{1}_{[0,Y(g)]}$ is the indicator function for the interval $[0,Y(g)]$.
We may rewrite the condition 
$$
\left\|\frac{ghc\bar{q}}{d}\right\|\le Y(g)
$$
as 
$$
u\equiv ghc\bar{q}\bmod{d} \quad \mbox{for some } u\in [-dY(g),dY(g)],
$$
that is, 
$$
uq\equiv ghc\bmod{d} \quad \mbox{for some } u\in [-dY(g),dY(g)]. 
$$
Thus, we can bound $E(D)$ by
\begin{align}\label{S1bound}
E(D)\ll \frac{LX}{qD} \cdot T(D)+O\lr{\frac{X^{2\varepsilon}DL}{G}},
\end{align}
where 
$$
T(D):=\sum\limits_{D/2\le d\le 2D} \ \sum\limits_{0<|h|\le X^{\varepsilon}qD/X} \sum\limits_{\substack{g|q}} \frac{1}{g} \sum\limits_{\substack{-2DY(g)\le u\le 2DY(g)\\ uq\equiv ghc\bmod{d}}} 1.
$$
We split $T(D)$ into two parts according to the cases $uq=ghc$ and $uq\neq ghc$, that is,
\begin{align}\label{T=T1+T2}
    T(D)=T_0(D)+T_1(D),
\end{align}
where
$$
T_0(D):=\sum\limits_{D/2\le d\le 2D}\ \sum\limits_{0<|h|\le X^{\varepsilon}qD/X}\ \sum\limits_{\substack{g|q\\g\le G}} \frac{1}{g} \sum\limits_{\substack{-2DY(g)\le u\le 2DY(g)\\ uq= ghc}} 1
$$
and 
$$
T_1(D):=\sum\limits_{D/2\le d\le 2D}\ \sum\limits_{0<|h|\le X^{\varepsilon}qD/X}\ \sum\limits_{\substack{g|q\\ g\le G}} \frac{1}{g}  \sum\limits_{\substack{-2DY(g)\le u\le 2DY(g)\\ uq\not= ghc\\ d|(uq-ghc)}} 1.
$$
Since $(c,q)=1$, from $uq=ghc$ it follows that $q/g$ divides $h$ and therefore,
$$
T_0(D)\le 2D\sum\limits_{\substack{g|q\\g\le G}} \frac{1}{g}  \sum\limits_{\substack{0<|h|\le X^{\varepsilon}qD/X\\ (q/g)|h}} 1\le 4\sum\limits_{\substack{g|q}} \frac{1}{g}\cdot \frac{X^\varepsilon gD^2}{X}  =  4\tau(q)\cdot \frac{X^{\varepsilon}D^2}{X}  \ll \frac{X^{2\varepsilon}D^2}{X}. 
$$
Further, moving the summation over $d$ inside, we have
\begin{equation*}
\begin{split}
T_1(D)\le & \sum\limits_{0<|h|\le X^{\varepsilon}qD/X} \sum\limits_{\substack{g|q\\g\le G}} \frac{1}{g} \ \sum\limits_{\substack{-2DY(g)\le u\le 2DY(g)\\ uq\not= ghc}} \tau(|uq-ghc|)\\ \ll & X^{\varepsilon}\sum\limits_{0<|h|\le X^{\varepsilon}qD/X} \sum\limits_{\substack{g|q\\g\le G}} \frac{1}{g}\cdot \left(1+DY(g)\right)\\
\ll & X^{3\varepsilon}\cdot \frac{qD}{X}\cdot \left(1+\frac{D}{L}\right),
\end{split}
\end{equation*}
where we have taken the definition of $Y(g)$ in \eqref{Ydef} into account. Now we fix $G:=LX^{-\varepsilon}$. Then
combining the above bounds for $T_0(D)$ and $T_1(D)$ with \eqref{S1bound} and \eqref{T=T1+T2}, and using $X=qL$, we get the bound
\begin{align} \label{bound-S1(D)}
E(D)&\ll \frac{LX}{qD}\lr{\frac{X^{2\varepsilon}D^2}{X}+\frac{X^{3\varepsilon}qD}{X}\lr{1+\frac{D}{L}}}+O\lr{\frac{X^{2\varepsilon}DL}{G}}\nonumber\\
&\ll X^{3\varepsilon}L^2 \cdot \lr{\frac{D}{X}+\frac{1}{L}+ \frac{D}{L^2}}.
\end{align}
Summing over $D=2^0,2^1, 2^2,\ldots\le \sqrt{X}$, we arrive at the following lemma.
\begin{lem}\label{lem_S4}
    Let $E$ be given as in \eqref{S4eq}. Then, 
\begin{align*}
E\ll X^{4\varepsilon} L^2 \left(X^{-1/2}+ L^{-1}+X^{1/2}L^{-2}\right).
\end{align*}
\end{lem}

\subsection{Estimation of $S_2$}
To complete the proof of Theorem \ref{main_thm}, it remains to evaluate the sum $S_2$, defined in \eqref{def-S2}.  
Interchanging the $m$- and $d$-summations in \eqref{def-S(D)}, and using $X\le md\le 2X$, we have 
\begin{align*}
   S(D)=4\sum_{\st{X/(2D)\le m\le 2X/D\\(m,q)=1}}\sum_{\st{b\in\mb{Z}\\(b,q)=1}}\Phi\lr{\frac{b}{L}}\sum_{\st{d\in\mb{Z}\\amd\equiv b \bmod{q}}}\chi_4(d)\Omega\lr{\frac{d}{D}}w\lr{\frac{md}{X}}.
\end{align*}
Since $\chi_4(d)=0$ if $d$ is even, we may restrict the $d$-sum  to odd integers. We divide this sum into two parts according to whether $d\equiv1 \bmod{4}$ or $d\equiv -1 \bmod{4}$. Writing $d=4k_1+1$ and $d=4k_2-1$, respectively, we have 
\begin{align}\label{S2=S5-S6}
    S(D)=S_+(D)-S_-(D), 
\end{align}
with 
\begin{align*}
    S_+(D)=&4\sum_{\st{X/(2D)\le m\le 2X/D\\(m,q)=1}}\sum_{\st{b\in\mb{Z}\\(b,q)=1}}\Phi\lr{\frac{b}{L}}\sum_{\substack{k_1\in\mb{Z}\\k_1 \equiv \bar{4}(\bar{a}\bar{m}b-1) \bmod{q}}}w\lr{\frac{m(4k_1+1)}{X}}\Omega\lr{\frac{(4k_1+1)}{D}}\\
   =& 4\sum_{\st{X/(2D)\le m\le 2X/D\\(m,q)=1}}\sum_{\substack{b\in\mb{Z}\\(b,q)=1}}\Phi\lr{\frac{b}{L}}\sum_{\substack{l_1\in\mb{Z}\\l_1 \equiv 1+4\bar{4}(\bar{a}\bar{m}b-1) \bmod{4q}}}w\lr{\frac{l_1m}{X}}\Omega\lr{\frac{l_1}{D}}
\end{align*}
and 
\begin{align*}
         S_-(D)=&4\sum_{\st{X/(2D)\le m\le 2X/D\\(m,q)=1}}\sum_{\substack{b\in\mb{Z}\\(b,q)=1}}\Phi\lr{\frac{b}{L}}\sum_{\substack{k_1\in\mb{Z}\\k_1 \equiv \bar{4}(\bar{a}\bar{m}b+1) \bmod{q}}}w\lr{\frac{m(4k_1-1)}{X}}\Omega\lr{\frac{(4k_1-1)}{D}}\\
        =& 4\sum_{\st{X/(2D)\le m\le 2X/D\\(m,q)=1}}\sum_{\substack{b\in\mb{Z}\\(b,q)=1}}\Phi\lr{\frac{b}{L}}\sum_{\substack{l_2\in\mb{Z}\\l_2 \equiv -1+4\bar{4}(\bar{a}\bar{m}b+1) \bmod{4q}}}w\lr{\frac{l_2m}{X}}\Omega\lr{\frac{l_2}{D}}.
\end{align*}
Here $\bar{4}$ is a multiplicative inverse of $4$ modulo $q$, which exists since $q$ is odd. 

For $X/(2D)\le m\le 2X/D$ and $t\in\mb{R}$, we define a test function $W_m$ with compact support in $[1/2,2]$
as
\begin{align*}
    W_m(t):=w\left(\frac{Dmt}{X}\right)\Omega\left(t\right)
\end{align*}
so that 
\begin{equation*} 
w\lr{\frac{lm}{X}}\Omega\lr{\frac{l}{D}}=W_m\left(\frac{l}{D}\right).
\end{equation*}
Now, using Poisson summation, Lemma \ref{Poisson}, we have 
\begin{equation*}
\begin{split}
\sum_{\substack{l_1\in\mb{Z}\\l_1 \equiv 1+4\bar{4}(\bar{a}\bar{m}b-1) \bmod{4q}}}W_m\left(\frac{l_1}{D}\right)
= &  \frac{D}{4q} \sum\limits_{h_1\in \mathbb{Z}} \hat{W}_m\left(\frac{h_1D}{4q}\right)e\left(h_1\cdot \frac{b \bar{4}\bar{a}\bar{m}}{q}\right)e\lr{h_1\cdot\frac{1-4\bar{4}}{4q}}.
\end{split}
\end{equation*}
Similarly,
\begin{equation*}
\begin{split}
\sum_{\substack{l_2\in\mb{Z}\\l_2 \equiv -1+4\bar{4}(\bar{a}\bar{m}b+1) \bmod{4q}}}W_m\left(\frac{l_2}{D}\right)
= &  \frac{D}{4q} \sum\limits_{h_2\in \mathbb{Z}}\hat{W}_m\left(\frac{h_2D}{4q}\right)e\left(h_2\cdot \frac{b \bar{4}\bar{a}\bar{m}}{q}\right)e\lr{h_2\cdot\frac{-1+4\bar{4}}{4q}}.
\end{split}
\end{equation*}
The contributions of $h_1=0$ and $h_2=0$ are equal and therefore cancel out in \eqref{S2=S5-S6}. Consequently, we are left with
\begin{align} \label{splitagain}
    S(D)=\tilde{S}_+(D)-\tilde{S}_-(D),
\end{align}
where
\begin{align*}
    \tilde{S}_+(D)&=\frac{D}{4q} \sum_{\st{X/2D\le m\le 2X/D\\(m,q)=1}}\ \sum\limits_{\st{h_1\in \mathbb{Z}\\h_1\neq 0}}\hat{W}_m\left(\frac{h_1D}{4q}\right) e\lr{h_1\cdot\frac{1-4\bar{4}}{4q}}\times\nonumber\\&\sum_{\substack{b\in\mb{Z}\\(b,q)=1}}\Phi\lr{\frac{b}{L}}e\left(b\cdot \frac{h_1 c_1\bar{m}}{q}\right)
\end{align*}
and 
\begin{align*}
    \tilde{S}_-(D)&=\frac{D}{4q} \sum_{\st{X/2D\le m\le 2X/D\\(m,q)=1}}\ \sum\limits_{\st{h_2\in \mathbb{Z}\\h_2\neq 0}}\hat{W}_m\left(\frac{h_2D}{4q}\right)e\lr{h_2\cdot\frac{-1+4\bar{4}}{4q}}\times\nonumber\\&\sum_{\substack{b\in\mb{Z}\\(b,q)=1}}\Phi\lr{\frac{b}{L}}e\left(b\cdot \frac{h_2 c_1\bar{m}}{q}\right).
\end{align*}
Here we assumed $\bar{4}\bar{a}\equiv c_1\bmod{q}$ with $1\leq c_1\leq q$. Using the triangle inequality and the rapid decay of $\hat{W}_m$ to truncate the $h_i$-sums at $|h_i|\le X^{\varepsilon}q/D$ at the cost of a negligible error, we bound $\tilde{S}_{\pm}(D)$ by
\begin{align}\label{def-S4(D)}
    \tilde{S}_{\pm}(D)\ll \frac{D}{q}\sum_{\st{X/2D\le m\le 2X/D\\(m,q)=1}}\ \sum\limits_{0<|h|\le X^{\varepsilon}q/D}\Big| \sum_{\substack{b\in\mb{Z}\\(b,q)=1}}\Phi\lr{\frac{b}{L}}e\left(b\cdot \frac{h c_1\bar{m}}{q}\right)\Big|.
\end{align}
The right-hand side of \eqref{def-S4(D)} is the same as that of \eqref{S1D} with the replacements
\begin{equation*}
\frac{X}{D}\leftrightarrow D, \quad m\leftrightarrow d,\quad c\leftrightarrow c_1. 
\end{equation*}
Consequently, a similar treatment as in subsection~\ref{sec-error} gives the bound 
\begin{align*}
\tilde{S}_{\pm}(D)\ll X^{3\varepsilon} L^2 \cdot \lr{\frac{1}{D}+\frac{1}{L}+\frac{X}{DL^2}}
\end{align*}
in place of  \eqref{bound-S1(D)}. 
Recalling that $X=qL$ and using 
\begin{equation*} 
     |S_2|\le \sum_{\sqrt{X}<D\le 2X}|\tilde{S}_+(D)|+ \sum\limits_{\sqrt{X}<D\le 2X} |\tilde{S}_-(D)|
\end{equation*} 
by \eqref{def-S2} and \eqref{splitagain}, where $D$ runs over powers of $2$, we arrive at the following lemma corresponding to Lemma \ref{lem_S4}. 

\begin{lem}\label{lem_S4new}
    Let $S_2$ be given as in \eqref{def-S2}. Then,
\begin{align*}
S_2\ll X^{4\varepsilon} L^2 \left(X^{-1/2}+ L^{-1}+X^{1/2}L^{-2}\right).
\end{align*}
\end{lem}

\subsection{Proof of Theorem~\ref{main_thm}} 
Recall from Theorem~\ref{main_thm} that $L=q^{\beta}$ and $X=q^{1+\beta}$ with $1/3< \beta<1$. Suppose without loss of generality that $4\varepsilon<\beta-1/3$. Then it follows that
\begin{align*}
    E, S_2\ll L^2X^{-\varepsilon}.
\end{align*}
Combining this with Lemma~\ref{lem_S3}, \eqref{S=S1+S2} and
\eqref{splitting}, we deduce that
\begin{align*}
    S=\pi\hat{\Phi}(0)\hat{w}(0)L^2\cdot\frac{\varphi(q)}{q}\cdot \prod_{p|q}\lr{1-\frac{\chi_4(p)}{p}}\left(1+o(1)\right)
\end{align*}
as $q\rightarrow \infty$. 
Also, by Lemma~\ref{lem_S3}, the above main term is bounded below by 
$$\frac{L^2}{(\log\log q)^2},$$
which dominates the error term. This completes the proof of Theorem~\ref{main_thm}.

\section{Proof of Corollary \ref{co_main_thm}}\label{sec_proof_cor}

We start the proof by the following lemma providing a bound for $r_2(n)$.
\begin{lem}\label{lem_bound_r_2(n)} For any given positive integer $n$, let $r_2(n)$ be the number of integral representations of $n$ as a  sum of two squares.  Then we have 
    \begin{align*}
        r_2(n)\ll n^{c/\log\log n},
    \end{align*}
    for some $c>0.$
\end{lem}
\begin{proof}
    From the convolution formula \eqref{convo}, we have  
    \begin{align*}
        r_2(n)\leq 4 \tau(n).
    \end{align*}
It is well-known that
    \begin{align*}
      \tau(n)\ll n^{\delta/\log\log n}
    \end{align*}
 for some constant $\delta>0$ (see \cite[Theorem 317]{HaWr}). The lemma follows.
\end{proof}
Let an irrational number $\alpha$ be given. Then by Lemma~\ref{lem_Dirichlet_approx}, there are infinitely many pairs $(a, q)\in \mathbb{N}\times \mathbb{Z}$ such that
\begin{align}\label{con_q_1}
    \left|\alpha-\frac{a}{q}\right|\le \frac{24 }{q^2}, \ (2a,q)=1.
\end{align}
This proves the last part of Corollary \ref{co_main_thm}.
Now, fix such a pair $(a,q)$ satisfying the above and suppose $L, X$ to be parameters satisfying
\begin{align*}
    L=q^\beta, \ X=Lq=q^{1+\beta},
\end{align*}
where $\beta$ and $\gamma$ are related by the equation
$$
\gamma=\frac{1-\beta}{1+\beta}.
$$
Note that since $0<\gamma<1/2$, we have 
$1/3<\beta<1$, and hence, Theorem \ref{main_thm} is applicable. If $n\in \mathbb{N}$ in the range $X\le n\le 2X$ satisfies $r_2(n)\geq 1$ and one of the congruences  
\begin{align*}
    an\equiv b \bmod{q}, \ |b|\leq L,\ (b,q)=1,
\end{align*}
then writing $an=b+n_1q$ for some integer $n_1$, it follows from \eqref{con_q_1} that
\begin{align*}
    \alpha n=n_1+\frac{b}{q}+\frac{\theta n}{q^2},
\end{align*}
where $|\theta|\le 24$. In other words, 
\begin{equation*}
    \|\alpha n\|\ll \frac{L}{q}+\frac{X}{q^2}=2q^{\beta-1}=2X^{-\gamma} \ll n^{-\gamma},
\end{equation*}
with $\gamma=(1-\beta)/(1+\beta)$.
Hence, such $n$ lies in the set
\begin{align*}
    \{X\le n\le 2X: n=x^2+y^2, x,y\in\mb{Z}, \ \|n\alpha\|<C_1n^{-\gamma}\}=:\mathcal{A}(\alpha, X), \ \text{ say},
\end{align*}
where $C_1>0$ is a suitable absolute constant. Now, recall that $\Phi$ and $w$ are supported in $[-1,1]$ and $[1,2]$, respectively. 
Thus, using Lemma~\ref{lem_bound_r_2(n)}, and noting that
\begin{align*}
    b_1\not\equiv b_2 \bmod{q} \Longrightarrow n_1\not\equiv n_2 \bmod{q}
\end{align*} 
if $q$ is large enough, 
we deduce from Theorem~\ref{main_thm} that
\begin{align*}
    \frac{L^2}{(\log\log q)^2}\ll S=\sum\limits_{\substack{b\in\mb{Z}\\(b,q)=1}}\Phi\lr{\frac{b}{L}} \sum\limits_{\substack{n\\ na\equiv b\bmod{q}}} r_2(n)w\lr{\frac{n}{X}}\ll X^{c/(\log\log X)} |\mathcal{A}(\alpha, X)|.
\end{align*}
Since $q=X^{1/(1+\beta)}$ and $L^2=X^{2\beta/(1+\beta)}=X^{1-\gamma}$, it follows that 
\begin{align*}
    |\mathcal{A}(\alpha, X)|\gg X^{1-\gamma-C_2/(\log\log X)}
\end{align*}
if $C_2>c$. This completes the proof of Corollary~\ref{co_main_thm}.

\end{document}